\documentclass[12pt,dvips]{amsart}
\usepackage{euler, amsfonts, amssymb, latexsym, epsfig,epic,stmaryrd}
\usepackage{hyperref}

\newtheorem{Theorem}{Theorem} 
 
\newtheorem{Lemma}{Lemma}

\newtheorem*{Corollary*}{Corollary}
 
\newtheorem*{Theorem*}{Theorem}
\theoremstyle{remark}
\newtheorem{Example}{Example}

\newcommand\Union{\bigcup}

\newcommand\integers{{\mathbb Z}}
\newcommand\rationals{{\mathbb Q}}

\newcommand\junk[1]{}

\theoremstyle{plain}

\renewenvironment{quotation}
{\list{}{
    \setlength\itemindent{0em}%
    \setlength\leftmargin{1.5em}
    \setlength\rightmargin{1.5em}
  }%
\item[]}
{\endlist}

\newcommand\dfn{\bf} 

\newcommand\kk{{\mathbb F}}

\begin{document}
\pagestyle{plain}

\newcommand\QQ{\rationals}
\newcommand\ZZ{\integers}
\title{Frobenius splitting and M\"obius inversion}

\author{Allen Knutson}
\thanks{Supported by an NSF grant.}
\email{allenk@math.cornell.edu}
\date{February 11, 2009}

\maketitle

\begin{abstract}
  We show that the fundamental class in $K$-homology of a Frobenius
  split scheme can be computed as a certain alternating sum over
  irreducible varieties, with the coefficients computed using M\"obius
  inversion on a certain poset.

  If $G/P$ is a generalized flag manifold and $X$ is an
  irreducible subvariety homologous to a multiplicity-free union
  of Schubert varieties, then using a result of Brion we show how
  to compute the $K_0$-class $[X]\in K_0(G/P)$ from the Chow class
  in $A_*(G/P)$. 
\end{abstract}

\tableofcontents

\section{Statement of results}

\newcommand\calP{{\mathcal P}}
Let $X$ be a Noetherian scheme, and let $\calP$ be a finite set of (irreducible)
subvarieties of $X$, with the following {\dfn intersect-decompose} property:
for any subset $S\subseteq \calP$, the geometric components of $\bigcap S$ 
should also be elements of $\calP$. (In particular, if $S=\emptyset$
we interpret $\bigcap S$ as $X$, and require that $\calP$ contain $X$'s
geometric components.) Let $\calP_X$ denote the obvious minimal such $\calP$,
constructed from $X$'s geometric components by intersecting and decomposing
until done.

\junk{
For example, let $X = X_{213} \cup X_{132} \subseteq \Flags(\kk^3)$,
where $X_{213} = \{ (V_1,V_2) : V_2 \supset \{[*,0,0]\} \}$
and $X_{132} = \{ (V_1,V_2) : V_1 \subset \{[*,*,0]\} \}$ are
the two components. Then their intersection has two components,
$$ C_1 = \{ (\{[*,0,0]\},V_2) \}, C_2 = \{ (V_1,\{[*,*,0]\}) \} $$
which further intersect in a point $\{ [*,0,0],[*,*,0] \}$.
So this $\calP_X$ has five elements, $\{X_{213},X_{132},C_1,C_2,C_1\cap C_2\}$.
}

For example, let $X = \{(x,y,z) : y(y z^2 - x^2(x-z)) = 0\}$.
This has two components, $A := \{y=0\}$ and $B := \{y z^2 = x^2(x-z))\} $.
Their (nonreduced) intersection is $\{y=x^2(x-z)=0\}$, 
which has geometric components $C := \{y=x=0\}$ and $D := \{y=x-z=0\}$. 
Finally, $C\cap D = \{\vec 0\}$. So $\calP_X = \{A,B,C,D,\{\vec 0\}\}$.

Note that in this example, even though $X$ was reduced (and even
Cohen-Macaulay) one ran into nonreducedness when one started intersecting
components. There is a well-known condition that allows one to 
avoid this:

\begin{Lemma}\label{lem:frobsplit}
  Let $X$ be Frobenius split (for which our reference
  is \cite{BK}). Then for any $A,\{B_i\} \in \calP_X$, $A\cap \bigcup_i B_i$ 
  is reduced.
\end{Lemma}

\begin{proof}
  This is immediate from \cite[proposition 1.2.1]{BK}.
\end{proof}

The {\dfn M\"obius function} of a finite poset $P$ is the
unique function $\mu_P : P \to \integers$ such that $\forall p\in P,
\sum_{q\geq p} \mu(q) = 1$. 

\begin{Theorem}\label{thm:mobius}
  Let $X$ be a reduced scheme such that for any $A,\{B_i\}\in \calP_X$, 
  $A\cap \bigcup_i B_i$ is reduced. 
  Let $\calP \supseteq \calP_X$ be a collection
  of subvarieties with the intersect-decompose property. 
  For each $A\in \calP$, let
  $[A] \in K_0(X)$ denote the $K$-homology class of the 
  structure sheaf of $A$. Then
  $$ [X] = \sum_{A \in \calP} \mu_{\calP}(A)\ [A]. $$
  (In fact $\mu_{\calP}(A) = 0$ unless $A\in \calP_X$, in which case
  $\mu_{\calP}(A) = \mu_{\calP_X}(A)$.)

  Assume now that $X$ carries an action of a group $G$.
  Assume too that $G$ preserves each element of $\calP$; this is
  automatic if $\calP = \calP_X$ and $G$ is connected.
  Then the classes in $G$-equivariant $K$-homology
  obey exactly the same formula above.
\end{Theorem}

It probably appears superfluous at this point to allow $\calP$ to
be any larger than $\calP_X$, insofar as it doesn't change the formula above.
The recursive definition of $\calP_X$ makes it difficult to compute,
however, and sometimes it is easier to give an upper bound.
For example, if $Y$ is a scheme carrying an action of a group $B$ with
finitely many orbits, and $X \subseteq Y$ is closed and $B$-invariant, 
then we can take $\calP$ to be the set of $B$-orbit closures contained in $X$.

In \cite{Brion} was proven the following remarkable fact:

\begin{Theorem}\label{thm:brion}
  Let $X$ be a subvariety (i.e. reduced and irreducible subscheme)
  of a generalized flag manifold $G/P$.
  Assume that the Chow class $[X]_{Chow} \in A(G/P)$ 
  is a sum of Schubert classes $\sum_{d\in D} [X_d]_{Chow}$, with no
  multiplicities. (Here $D$ is a subset of the Bruhat order $W/W_P$.)

  Then there is a flat degeneration of $X$ to the reduced union
  $\Union_{d\in D} X_d$, and both subschemes are Cohen-Macaulay.
\end{Theorem}

Combining this with the theorem above, we will obtain

\begin{Theorem}\label{thm:mobiusbrion}
  Let $X$ be a multiplicity-free subvariety of $G/P$, in the 
  sense of \cite{Brion}, with $[X]_{Chow} = \sum_{d\in D} [X_d]_{Chow}$.
  Let $\calP \subseteq W/W_P$ be the set of Schubert varieties contained
  in $\cup_{d\in D} X_d$ (an order ideal in the Bruhat order on $W/W_P$).
  Then as an element of $K_0(G/P)$,
  $$ [X] = \sum_{X_e \subseteq \bigcup_{d\in D} X_d} \mu_{\calP}(X_e)\ [X_e]. $$
\end{Theorem}

Note that the $X$ in the last theorem above is {\em not} assumed to be
Frobenius split. (Its degeneration $\bigcup_{d\in D} X_d$ is, automatically
\cite[theorem 2.2.5]{BK}.)

The preprint \cite{Snider} applies our theorem \ref{thm:mobiusbrion} to 
the case that $X$ is a multiplicity-free Richardson variety in
a Grassmannian, giving an independent proof of Buch's $K$-theoretic
Littlewood-Richardson rule \cite{Buch} in the case that the
ordinary product is multiplicity-free.

In \cite{KLS} we will use theorem \ref{thm:mobius} to compute the
$K$-classes of the closed strata in the cyclic Bruhat decomposition,
whose study was initiated in \cite{Sasha} and continued in e.g. \cite{W,PSW,LW}.

\subsection*{Acknowledgements}
We thank Michelle Snider for many useful conversations, and most
especially for the insight that the second half of lemma \ref{lem:largerposet} 
should be traced to the first.

\section{Proofs}

We first settle the difference between the poset $\calP_X$ and more general
posets $\calP$, with a combinatorial lemma we learned from Michelle Snider.

\begin{Lemma}\label{lem:largerposet}
  Let $P \supseteq Q$ be two finite posets such that
  $\forall S\subseteq Q$, all the greatest lower bounds in $P$ of $S$
  are also in $Q$. (In particular, the $S = \emptyset$ case implies
  that $Q$ contains all of $P$'s maximal elements.)
  Then
  \begin{enumerate}
  \item for each $p\in P\setminus Q$, the set
    $Q_P = \{q \in Q : q\geq p\}$ has a unique minimal element, and
  \item $\mu_P(p) = \mu_Q(p)$ for $p\in Q$, and otherwise $\mu_P(p) = 0$. 
  \end{enumerate}
\end{Lemma}

\begin{proof}
  \begin{enumerate}
  \item 
    Let $p\notin Q$. 
    Let $S =\{s \in P : \forall q\in Q_P, q\geq s\}$. 
    Tautologically, $Q_P$ is an upward order ideal, $S$ is a
    downward order ideal, and $S \ni p$. 
    By assumption, $Q$ contains the maximal elements of $S$.
    Pick one that is larger than $p$ and call it $q_{min}$.
    
    Since $q_{min}\in S$, $q_{min} \leq q'$ for all $q'\in Q_P$. 
    Since $q_{min}\geq p$ and $q_{min} \in Q$, $q_{min}\in Q_P$. 
    So $q_{min}$ is the unique minimal element of $Q_P$.
  \item   
    Define $m:P\to\integers$ by 
    $m(p) = \mu_Q(p)$ for $p\in Q$, and otherwise $m(p) = 0$.
    Our goal is to show that $\mu_P = m$, or equivalently, that $m$
    satisfies the defining criterion of M\"obius functions:
    $\forall p\in P$, $\sum_{p'\in P, p'\geq p} m(p') = 1$.

    Let $q_{min} \geq p$ be the minimum element of $Q_p$.
    (It equals $p$ iff $p\in Q$.)
    Then
    $$ \sum_{p'\in P, p'\geq p} m(p') 
    = \sum_{p'\in Q, p'\geq p} m(p') 
    = \sum_{p'\in Q_p} m(p') 
    = \sum_{p'\in Q, p'\geq q_{min}} m(p') 
    = \sum_{p'\in Q, p'\geq q_{min}} \mu_Q(p')
    = 1. $$
  \end{enumerate}
\end{proof}

The following lemma establishes the property of M\"obius functions
that we will use to connect them to $K$-classes.

\begin{Lemma}\label{lem:othermobius}
  \begin{enumerate}
  \item 
    Let $P$ be a finite poset, and $Q$ a downward order ideal.
    Extend $\mu_Q$ to $P$ by defining $\mu_Q(p) = 0$ for $p\in P\setminus Q$.
    Then $\sum_{p'\geq p} \mu_Q(p') = [p\in Q]$, 
    meaning $1$ for $p\in Q$, $0$ for $p\notin Q$.
  \item 
    Let $P$ be a finite poset, with two downward order ideals $P_1,P_2$ such 
    that $P = P_1 \cup P_2$. Extend $\mu_{P_1},\mu_{P_2},\mu_{P_1\cap P_2}$
    to functions on $P$ by defining them as $0$ on the new elements. Then
    $ \mu_P = \mu_{P_1} + \mu_{P_2} - \mu_{P_1\cap P_2}$. 
  \end{enumerate}
\end{Lemma}

\begin{proof}
  \begin{enumerate}
  \item For any $p\in P$,
    $$ \sum_{p'\in P, p'\geq p} \mu_Q(p') = \sum_{q\in Q, q\geq p} \mu_Q(q) $$
    which is an empty sum unless $p\in Q$. If $p\in Q$, then it becomes
    the usual M\"obius function sum for $Q$, so adds up to $1$.
  \item 
    By the result above,
    $$ \sum_{q\geq p} 
    \left(\mu_{P_1}(q) + \mu_{P_2}(q) - \mu_{P_1\cap P_2}(q)\right) 
    = [q \in P_1] + [q\in P_2] - [q \in P_1 \cap P_2]. $$
    If $q\in P_1 \setminus P_2$, this gives $1+0-0 = 1$;
    similarly if $q\in P_2\setminus P_1$.
    If $q\in P_1\cap P_2$, this gives $1+1-1 = 1$.
    These are all the cases, by the assumption $P = P_1 \cup P_2$.

    Since
    $$ \sum_{q\geq p} 
    \left(\mu_{P_1}(q) + \mu_{P_2}(q) - \mu_{P_1\cap P_2}(q)\right) = 1$$
    for all $p\in P$, this $\mu_{P_1} + \mu_{P_2} - \mu_{P_1\cap P_2}$
    must be the M\"obius function $\mu_P$.
  \end{enumerate}
\end{proof}

\junk{
\begin{Theorem}\label{thm:mobius}
  Let $X$ be a scheme such that for any $A,B\in \calP_X$, 
  $A\cap B$ is reduced. Let $\calP \supseteq \calP_X$ be a collection
  of subvarieties with the intersect-decompose property. 
  For each $A\in \calP$, let
  $[A] \in K_0(X)$ denote the $K$-homology class of the 
  structure sheaf of $A$. Then
  $$ [X] = \sum_{A \in \calP} \mu_{\calP}(A)\ [A]. $$
  (In fact $\mu_{\calP}(A) = 0$ unless $A\in \calP_X$, in which case
  $\mu_{\calP}(A) = \mu_{\calP_X}(A)$.)
\end{Theorem}

  Let $P \supseteq Q$ be two finite posets such that
  $\forall S\subseteq Q$, all the greatest lower bounds in $P$ of $S$
  are also in $Q$. (In particular, the $S = \emptyset$ case implies
  that $Q$ contains all of $P$'s maximal elements.)
}

\begin{proof}[Proof of theorem \ref{thm:mobius}.]
  First, we observe that $\calP_X \subseteq \calP$ satisfies the
  condition of lemma \ref{lem:largerposet}; for any collection $S$
  of varieties in $\calP_X$, and $Y \in \calP$ 
  such that $Y \subseteq \bigcap S$, there exists $Y' \in \calP_X$,
  $Y' \supseteq Y$. Proof: since $Y$ is irreducible, it is contained
  in some geometric component $Y'$ of $\bigcap S$, and by the
  recursive definition of $\calP_X$ we know $Y' \in \calP_X$.

  By part (2) of lemma \ref{lem:largerposet}, 
  $$ \sum_{A \in \calP} \mu_{\calP}(A)\ [A]
  = \sum_{A \in \calP_X} \mu_{\calP_X}(A)\ [A]. $$
  So it suffices for the remainder to assume that $\calP = \calP_X$.

  If $X$ is irreducible, then $\calP_X = \{X\}$, and the formula is
  easily verified:
  $$ \sum_{A \in \calP_X} \mu_{\calP_X}(A)\ [A] = \mu_{\calP_X}(X)\ [X] 
  = 1\ [X] = [X]. $$
  This will be the base of an induction on the number of components;
  we assume hereafter that there are at least $2$.

  Let $A$ be a geometric component of $X$, and $X'$ the union of
  the other components. Then we have a formula on $K$-homology classes:
  \begin{equation}
    \label{eq:Khomology}
     [X] = [A] + [X'] - [A\cap X']. 
  \end{equation}
  Let $P_1 = \{Y \in \calP_X : Y \subseteq A \}$,
  $P_2 = \{Y \in \calP_X : Y \subseteq X' \}$.
  Then by induction, the three terms on the right-hand side can
  be computed by M\"obius inversion on $P_1,P_2,P_1\cap P_2$.

  Now apply part (2) of lemma \ref{lem:othermobius} to say that
  $$ \mu_{\calP_X} = \mu_{P_1} + \mu_{P_2} - \mu_{P_1\cap P_2}. $$
  Putting these together,
  \begin{eqnarray*}
     \sum_{C \in \calP_X} \mu_{\calP_X}(C)\ [C]
  &=&  \sum_{C \in \calP_X}  
  (\mu_{P_1}(C) + \mu_{P_2}(C) - \mu_{P_1\cap P_2}(C))\ [C] \\
  &=&  \left( \sum_{C \in P_1} \mu_{P_1}(C)\ [C] \right)
  +  \left( \sum_{C \in P_2} \mu_{P_2}(C)\ [C] \right)
  -  \left( \sum_{C \in P_1\cap P_2} \mu_{P_1\cap P_2}(C)\ [C] \right) \\
  &=& [A] + [X'] - [A\cap X'] \\
  &=& [X]. 
  \end{eqnarray*}
  If we intersect $G$-invariant subvarieties of $X$, the result is
  again $G$-invariant. If $G$ is connected, hence irreducible, then
  it preserves each component of any $G$-invariant subvariety.
  Hence by induction $G$ preserves each element of $\calP_X$.
  $G$-equivariant $K$-homology also satisfies equation (\ref{eq:Khomology}),
  and the remainder of the argument is the same.
\end{proof}

\junk{
\begin{Theorem}\label{thm:mobiusbrion}
  Let $X$ be a multiplicity-free subvariety of $G/P$, in the 
  sense of \cite{Brion}, with $[X]_{Chow} = \sum_{d\in D} [X_d]_{Chow}$.
  Let $\calP \subseteq W/W_P$ be the set of Schubert varieties contained
  in $\cup_{d\in D} X_d$ (an order ideal in the Bruhat order on $W/W_P$).
  Then as an element of $K_0(G/P)$,
  $$ [X] = \sum_{X_e \subseteq \bigcup_{d\in D} X_d} \mu_{\calP}(X_e)\ [X_e]. $$
\end{Theorem}
}

\begin{proof}[Proof of theorem \ref{thm:mobiusbrion}.]
  The $K$-class is preserved under flat degenerations, so
  $ [X] = \left[ \bigcup_{d\in D} X_d \right]. $
  By \cite[theorem 2.2.5]{BK}, there is a Frobenius splitting on $G/P$
  for which $\bigcup_{d\in D} X_d$ is compatibly split. In particular,
  $\bigcup_{d\in D} X_d$ is Frobenius split, 
  and lemma \ref{lem:frobsplit} applies.

  To apply theorem \ref{thm:mobius}, we need a collection $\calP$ of
  irreducible subvarieties of $\bigcup_{d\in D} X_d$, with the
  intersect-decompose property. So we take $\calP$ to be the set of
  Schubert varieties $\{X_e\}$ contained in $\bigcup_{d\in D} X_d$.
  Since the Schubert varieties are the orbit closures for the action of
  a Borel subgroup on $G/P$, any intersection $A \cap \bigcup_i B_i$
  will again be Borel-invariant. Since that Borel acts with finitely
  many orbits, any Borel-invariant subvariety is an orbit closure.
  This shows that the components of any intersection $A \cap \bigcup_i B_i$
  are in $\calP$.

  Now we apply theorem \ref{thm:mobius}, and obtain the desired formula.
\end{proof}

In the application in \cite{Snider}, the subvariety $X$ is preserved
under the action of the maximal torus $T$ of $G$, and of course the Schubert
varieties $\{X_d\}$ are as well.
However, theorem \ref{thm:mobiusbrion} does {\em not} give an equality
of $T$-equivariant $K$-homology classes, as the flat degeneration is
not $T$-equivariant.

\bibliographystyle{alpha}    

\end{document}